\documentclass[runningheads]{llncs}

\usepackage[utf8]{inputenc}
 \usepackage{color}
 \usepackage{amsmath,amssymb, boxedminipage}

\usepackage{tikz}
\usetikzlibrary{fit}
\usetikzlibrary{shapes}
\usepackage{comment}

\usepackage{diagbox}
\usepackage{makecell}

\usepackage{amsthm}

\begin{document}

\title{Beyond recognizing well-covered graphs}

\author{Carl Feghali \and Malory Marin
 \and R\'emi Watrigant}
 
 \institute{Univ Lyon, CNRS, ENS de Lyon, Université Claude Bernard Lyon 1, LIP UMR5668, France\\
\email{firstname.lastname@ens-lyon.fr}}

\maketitle

\begin{abstract}
 We prove a number of results related to the computational complexity of recognizing well-covered graphs.  Let $k$ and $s$ be positive integers and let $G$ be a graph. Then $G$ is said \begin{itemize}
     \item $\mathbf{W_k}$ if for any $k$ pairwise disjoint independent vertex sets $A_1, \dots, A_k$ in $G$, there exist $k$ pairwise disjoint maximum independent sets $S_1, \dots,S_k$ in $G$ such that $A_i \subseteq S_i$ for $i \in [k]$.
     \item $\mathbf{E_s}$ if every independent set in $G$ of size at most $s$ is contained in a maximum independent set in $G$. 
 \end{itemize}
Chvátal and Slater (1993) and Sankaranarayana and Stewart (1992) famously showed that recognizing $\mathbf{W_1}$ graphs or, equivalently, well-covered graphs is coNP-complete.  We extend this result by showing that recognizing $\mathbf{W_{k+1}}$ graphs in either $\mathbf{W_k}$ or $\mathbf{E_s}$ graphs is coNP-complete. This answers a question of Levit and Tankus (2023) and strengthens a theorem of Feghali and Marin (2024). We also show that recognizing $\mathbf{E_{s+1}}$ graphs is $\Theta_2^p$-complete even in $\mathbf{E_s}$ graphs, where $\Theta_2^p = \text{P}^{\text{NP}[\log]}$ is the class of problems solvable in polynomial time using a logarithmic number of calls to a \textsc{Sat} oracle. This strengthens a theorem of Berg\'e, Busson, Feghali and Watrigant (2023). We also obtain the complete picture of the complexity of recognizing chordal $\mathbf{W_k}$ and
$\mathbf{E_s}$ graphs which, in particular, simplifies and generalizes a result of Dettlaff, Henning and Topp (2023).
\end{abstract}

\section{Introduction}

In a graph, an \emph{independent set} is a set of pairwise non-adjacent vertices. An independent set is said to be \emph{maximum} if it is of maximum size, and \emph{maximal} if it is not a subset of any other independent set. Merging these two notions, Plummer~\cite{PLUMMER197091} introduced the class of \emph{well-covered} graphs, defined as those graphs in which every maximal independent set is also maximum. Well-covered graphs have since been extensively studied; for instance, the problem of recognizing well-covered graphs is coNP-complete on general graphs~\cite{CHVATAL1993179,Sankaranarayana}, but polynomial-time solvable on
claw-free graphs \cite{TANKUS}, bipartite graphs~\cite{RavindraWell}, graphs with girth at least $5$~\cite{Finbow}, graphs with no cycle of length $4$ and $5$~\cite{FinbowC4C5}, graphs of bounded degree~\cite{Caro}, chordal graphs~\cite{Prisner}, cographs~\cite{Klein2013}, and perfect graphs with bounded clique number \cite{DEAN199467}. 
See the seminal survey by Plummer for more details \cite{Plummer1993}.

In this paper, we continue the study of well-covered graphs from a complexity viewpoint, though in a broader context. 
Staples introduced in 1975~\cite{Sta1} the hierarchy $\mathbf{W_k}$, as
a generalization of well-covered graphs. For a positive integer $k$, a graph $G$ is $\mathbf{W_k}$ if for any $k$ pairwise disjoint independent vertex subsets $A_1, \dots, A_k$ in $G$, there exist $k$ pairwise disjoint maximum independent sets $S_1, \dots, S_k$ in $G$ such that $A_i \subseteq S_i$ for $i \in [k]$. Observe, by definition, that a graph is $\mathbf{W_1}$ if and only if it is well-covered, and that $\mathbf{W_{k+1}} \subseteq \mathbf{W_k}$ for each $k \geqslant 1$. 
The class of $\mathbf{W_2}$ graphs was studied more specifically due to its relation to \emph{shedding vertices}: a vertex $v$ in a graph $G$ is \textit{shedding} if there does not exist an independent set in $V (G) \setminus N [v]$ which dominates $N(v)$. In particular, a graph without isolated vertices is $\mathbf{W_2}$ if, and only if, all of its vertices are shedding \cite{levit2017w2}. Shedding vertices were also used in order to characterize well-covered graphs of high girth~\cite{Finbow} (they are called \textit{extendable vertices}). 
Deciding whether a vertex of a graph is shedding is coNP-complete \cite{levit2023recognizing}, even in triangle-free graphs.  Similarly, recognizing $\mathbf{W_k}$ graphs was recently shown to be coNP-complete  (for each fixed $k\geqslant 2$)~\cite{feghali2023three}, answering a recent question from \cite{levit2023recognizing}. However, the more challenging complexity problems (also raised in \cite{levit2023recognizing}) of recognizing $\mathbf{W_2}$ graphs and shedding vertices when the input graph is well-covered remained open.

\begin{problem}[\cite{levit2023recognizing}]\label{problem1}
Determine the complexity of recognizing $\mathbf{W_2}$-graphs when the input graph is well-covered.  
\end{problem}

\begin{problem}[\cite{levit2023recognizing}]\label{problem2}
Determine the complexity of recognizing shedding vertices when the input graph is well-covered.  
\end{problem}

In this paper, we resolve Problems \ref{problem1} and \ref{problem2} in the following theorems. 

\begin{theorem}\label{thm:wk}
Let $k \geqslant 2$ and $G$ be $\mathbf{W_{k-1}}$. Then deciding if  $G$ is $\mathbf{W_k}$ is coNP-complete. 
\end{theorem}

\begin{theorem}\label{thm:shedding}
Deciding whether a vertex is shedding in a well-covered graph is coNP-complete. 
\end{theorem}

The proofs of these can be found in Section \ref{section:2}.

A related class of graphs, namely B-graphs, was introduced by Berge \cite{BERGE198231} to study well-covered graphs. It consists of graphs in which every vertex belongs to a maximum independent set (MIS). As recently observed \cite{berge20231}, B-graphs play a key role in Wi-Fi networks: under saturation, the throughput of an access point of a Wi-Fi network is proportional to the number of MIS this access point belongs to in the conflict graph. In particular, if the conflict graph is a B-graphs, it ensures that no access points will have its throughput close to zero.

In \cite{DEAN199467}, the class of B-graphs was generalized to that of \emph{$k$-extendable graphs}; for a positive integer $k$, a graph is $k$-extendable  if every independent set of size (exactly)
$k$ is contained in an MIS. Thus, B-graphs are exactly the $1$-extendable graphs and a graph $G$ is well-covered if and only if it is $k$-extendable for every $k \in \{1, 2, \dots,  |V(G)|\}$. 
This motivates an extension of the $\mathbf{W_k}$ hierarchy via $k$-extendability. Namely, if we denote by $\mathbf{E_s}$ the class of $k$-extendable graphs for all $k \in [s]$, we can infer that $\mathbf{E_{s+1}} \subseteq \mathbf{E_s}$ and $\mathbf{W_1} \subseteq \mathbf{E_s}$.
At the same time, there exist graphs which are $k$-extendable but not $(k-1)$-extendable (for instance, a graph with a universal vertex and different from a clique might be $2$-extendable, but cannot be $1$-extendable), hence 
the class $\mathbf{E_s}$ and the class of $s$-extendable graphs ($s$ fixed) are not identical. 
However, notice that in a $s$-extendable graph, where $s$ is fixed, testing $k$-extendability for $k\leqslant s$ can be done in polynomial time. This is achieved by checking if any independent set of size $k$ is contained in an independent set of size $s$. Dean and Zito~\cite{DEAN199467} showed that in certain classes of graphs (namely: bipartite graphs and perfect graphs with bounded clique number), if the graph is $k$-extendable for ``small values" of $k$ then it is well-covered. 
It is thus natural to ask whether testing well-coveredness becomes easier if the input graph is $\mathbf{E_s}$.  
Unfortunately, we demonstrate that this is not the case, as stated in our next theorem, which we prove in Section \ref{section:3}.

\begin{theorem}
    Let $k \geqslant 2$ and $s\geqslant 1$ and $G$ be $\mathbf{E_s}$. Then deciding whether $G$ is $\mathbf{W_k}$ is coNP-complete. 
\end{theorem}

We should note that, unlike recognizing $\mathbf{W_1}$ graphs (which are in coNP~\cite{CHVATAL1993179,Sankaranarayana}, where a polynomial certificate for \textit{not} being well-covered consists of two maximal independent sets of different sizes), there is no known polynomial certificate for being $1$-extendable, despite recognizing $1$-extendable graphs being NP-hard \cite{berge20231}, even on subcubic planar graphs and on unit disk graphs. Moreover, a close examination of the same reduction also establishes that recognizing $1$-extendable graphs is coNP-hard (see Section \ref{section:3}). So unless NP = coNP, this problem is neither in NP nor in coNP.  In our next result, whose proof is given in Section \ref{section:3}, we determine the precise complexity of this problem. 

\begin{theorem}\label{thm:es}
Let $s \geqslant 1$ and $G$ be $\mathbf{E_{s-1}}$. Then deciding whether $G$ is $\mathbf{E_s}$ is $\Theta_2^p$-complete. 
\end{theorem} 

Notice that $\Theta_2^p = \text{P}^{\text{NP}[\log]}$ is the class of problems that can be solved by a polynomial time algorithm using a logarithmic number of calls to a \textsc{Sat} oracle.

\medskip

For tractable graph classes, we note that a characterization of chordal well-covered graphs was established in \cite{Prisner}, stating that \emph{a chordal graph $G$ is well-covered if and only if the simplices of $G$ form a partition of $V(G)$}, which, in turn, leads to a linear-time algorithm for deciding if a chordal graph is well-covered. In \cite{berge20231}, the authors asked whether an analogous characterization for $1$-extendable graphs can be established. This was recently settled in~\cite{DetHeTop23} using a new object called \textit{successive clique covers}. Unfortunately, this characterization does not yield a linear-time algorithm. In our next theorem, whose proof is in Section \ref{sec:chordal}, we simplify their characterization, and use it to describe a linear-time algorithm. 

\begin{theorem}\label{thm:chordalmain}
A chordal graph $G$ is $1$-extendable if, and only if, there exists a partition of $V(G)$ into $\alpha(G)$ parts such that each part is a maximal clique of $G$. This can be tested in linear time.
\end{theorem}

Finally, in Section \ref{sec:conclusion}, we conclude the paper with some old and new open problems.

\section{Recognizing $\mathbf{W_k}$ graphs}\label{section:2}

We start by proving simultaneously Theorem \ref{thm:shedding} and the  case $k=2$ of Theorem \ref{thm:wk}, thereby resolving both Problems \ref{problem1} and \ref{problem2}.

\begin{figure}[h!]
\centering
\begin{tikzpicture}
\node[draw, circle, minimum size = 0.7cm] (K) at (-4,4.5) {$v$};

\node[draw, circle, minimum size = 0.7cm] (c1) at (-2,6) { \small $v_1$};
\node[draw, circle, minimum size = 0.7cm] (c2) at (-2,5) { \small $v_2$};
\node[draw, circle, minimum size = 0.7cm] (cm) at (-2,3) {\small $v_m$};
\node[draw, ellipse, minimum height = 5cm, minimum width = 1.6cm] (C) at (-2,4.5) {} ;

\node[draw, circle, minimum size = 0.7cm] (x11) at (1,9) {};
\node[draw, circle, minimum size = 0.7cm] (x12) at (2,9) {};
\node[draw, ellipse, minimum height = 1.2cm, minimum width = 2.5cm] (X1) at (1.5,9) {} ;
\node[minimum size = 0.7cm] () at (3.2,9) { $x_1$};

\node[draw, circle, minimum size = 0.7cm] (nx11) at (1,7.5) {};
\node[draw, circle, minimum size = 0.7cm] (nx12) at (2,7.5) {};
\node[draw, ellipse, minimum height = 1.2cm, minimum width = 2.5cm] (NX1) at (1.5,7.5) {} ;
\node[minimum size = 0.7cm] () at (3.2,7.5) { $\overline{x_1}$};

\node[draw, circle, minimum size = 0.7cm] (x21) at (1,5.5) {};
\node[draw, circle, minimum size = 0.7cm] (x22) at (2,5.5) {};
\node[draw, ellipse, minimum height = 1.2cm, minimum width = 2.5cm] (X2) at (1.5,5.5) {} ;
\node[minimum size = 0.7cm] () at (3.2,5.5) { $x_2$};

\node[draw, circle, minimum size = 0.7cm] (nx21) at (1,4) {};
\node[draw, circle, minimum size = 0.7cm] (nx22) at (2,4) {};
\node[draw, ellipse, minimum height = 1.2cm, minimum width = 2.5cm] (NX2) at (1.5,4) {} ;
\node[minimum size = 0.7cm] () at (3.2,4) { $\overline{x_2}$};

\node[draw, circle, minimum size = 0.7cm] (xn1) at (1,1.5) {};
\node[draw, circle, minimum size = 0.7cm] (xn2) at (2,1.5) {};
\node[draw, ellipse, minimum height = 1.2cm, minimum width = 2.5cm] (Xn) at (1.5,1.5) {} ;
\node[minimum size = 0.7cm] () at (3.2,1.5) { $x_n$};

\node[draw, circle, minimum size = 0.7cm] (nxn1) at (1,0) {};
\node[draw, circle, minimum size = 0.7cm] (nxn2) at (2,0) {};
\node[draw, ellipse, minimum height = 1.2cm, minimum width = 2.5cm] (NXn) at (1.5,0) {} ;
\node[minimum size = 0.7cm] () at (3.2,0) { $\overline{x_n}$};

\draw[fill = black] (1.5,2.6) circle (0.02);
\draw[fill = black] (1.5,2.8) circle (0.02);
\draw[fill = black] (1.5,3) circle (0.02);

\draw[fill = black] (-2,3.8) circle (0.02);
\draw[fill = black] (-2,4) circle (0.02);
\draw[fill = black] (-2,4.2) circle (0.02);

\draw(K) -- (c1) ;
\draw(K) -- (c2) ;
\draw(K) -- (cm) ;

\draw (x11) -- (x12) ;
\draw (nx11) -- (nx12) ;
\draw (x11) -- (nx12) ;
\draw (x11) -- (nx11) ;
\draw (x12) -- (nx11) ;
\draw (x12) -- (nx12) ;

\draw (x21) -- (x22) ;
\draw (nx21) -- (nx22) ;
\draw (x21) -- (nx22) ;
\draw (x21) -- (nx21) ;
\draw (x22) -- (nx21) ;
\draw (x22) -- (nx22) ;

\draw (xn1) -- (xn2) ;
\draw (nxn1) -- (nxn2) ;
\draw (xn1) -- (nxn2) ;
\draw (xn1) -- (nxn1) ;
\draw (xn2) -- (nxn1) ;
\draw (xn2) -- (nxn2) ;

\draw (c1) to[out = 45, in = 180] (X1) ;
\draw (c1) to[out = 0, in = 180] (NX2) ;
\draw (c1) to[out = -45, in = 180]  (Xn) ;
\draw (c2) to[out = 45, in = 180]  (NX1) ;
\draw (c2) to[out = 0, in = 180]  (X2) ;
\draw (cm) to[out = -45, in = 180]  (NXn) ;

\end{tikzpicture}
\caption{The graph $G'$.} \label{fig:W2W1}
\end{figure}
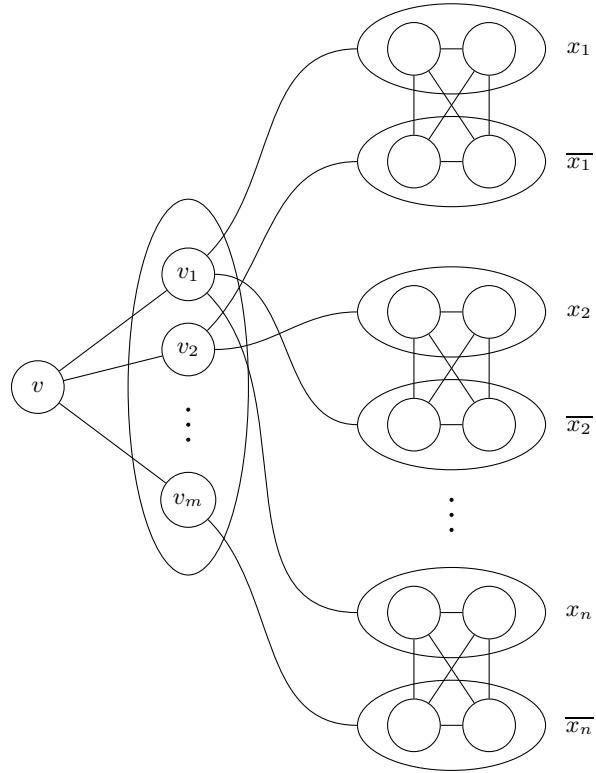

\begin{theorem}\label{thm:wcpartial}
Let $G$ be a well-covered graph and $v\in V(G)$.  Then the problems of deciding if $G$ is $\mathbf{W_2}$ and if $v$ is shedding are both coNP-complete.
\end{theorem}

\begin{proof}
From \cite{feghali2023three}, recognizing $\mathbf{W_2}$ graphs is in coNP. And deciding if a vertex is shedding is also in coNP, where a certificate consists of an independent set $S$ of $V (G) \setminus N [v]$ that dominates $N(v)$. 

To show the hardness for both problems, we describe a generic reduction from the complement of \textsc{3-Sat}.  Let $\Phi$ be a 3-CNF formula with a set $C$ of $m$ clauses $c_1,...,c_m$ and a set $X$ of $n$ variables $x_1,...,x_n$. Consider the graph $G$ (with $m+1+4n$ vertices) that is obtained from $\Phi$ as follows:
\begin{itemize}
\item create a clique $K$ of size $m$ with vertex labels $v_1, \dots, v_m$;
\item add a new vertex $v$ adjacent to all vertices of $K$;
\item for each variable $x_i$, create vertices $x_i^1$, $x_i^2$, $\overline{x_i}^1$, $\overline{x_i}^2$, and add all (six) possible edges between them;
\item if $x_i$ (resp. $\overline{x_i}$) appears in clause $c_j$, add the edges $x_i^1v_j$ and $x_i^2v_j$ (resp. $\overline{x_i}^1v_j$ and $\overline{x_i}^2v_j$).
\end{itemize}
That completes the construction of $G$; see also Figure \ref{fig:W2W1} for an illustration. 

We must show that $G$ is well-covered and that $\Phi$ is unsatisfiable if, and only if, $G$ is $\mathbf{W_2}$. For $1\leq i \leq n$, write $K^i = \{\ell^b\}_{\ell \in \{x_i, \overline{x_i}\}, b \in \{1,2\} }$ and note that $K^i$ is, by construction, isomorphic to the complete graph $K_4$ on four vertices. 

\begin{claim}
$G$ is well-covered.
\end{claim}

\begin{proof}
 Any independent set of $G$ can contain at most vertex from $\{v\} \cup K$ and at most one vertex from each $K^i$. Thus, $\alpha(G) \leqslant n+1$. However, since we can construct an independent set of size $n+1$ by taking $\{v\} \cup \{x_i^1 \}_{1\leqslant i \leqslant n}$, we have that $\alpha(G) = n+1$. So, to prove the claim, if we let $S$ be any maximal independent set of $G$, we must show that $S$ has size $n+1$.

If $S$ contains $v$, then $S-\{v\}$ is a maximal independent set of $G- (K\cup \{v\})$, which is exactly a disjoint union of $n$ copies of $K_4$. Thus, $S$ has size $n+1$. So we may assume that $S$ does not contain $v$. In this case, it must contain, by maximality and construction, a vertex $v_j \in V(K)$ for some $j \in [m]$. Consider the literals $\ell_1,\ell_2,\ell_3$ appearing in $c_j$. Then $S-\{v_j\}$ is maximal in $G-(K \cup \{v\} \cup \{\ell_i^b\}_{i \in \{1,2,3\}, b \in \{1,2\}})$, and, as it consists of a disjoint union of $n-3$ copies of $K_4$ and three copies of $K_2$,  $S$ has again size $n+1$, as required.  
\end{proof}

We first show that $\Phi$ is unsatisfiable if, and only if, $G$ is $\mathbf{W_2}$. 

First, suppose that $G$ is $\mathbf{W_2}$. For any truth assignment to $X$, for $1\leqslant i \leqslant n$, define a new variable $\ell_i$ by $\ell_i = x_i$ if $x_i$ is set to true, and by $\ell_i = \overline{x_i}$ otherwise. Let $S_1 = \{\ell_i^1\}_{1\leqslant i \leqslant n}$ and $S_2 = \{v\}$ and note that these form independent sets in $G$. Now, since $G\in \mathbf{W_2}$, $S_1$ is contained in an MIS $I$ of $G$ and, by our choice of $S_2$, $I = S_1 \cup \{v_j\}$ for some $j \in [m]$. Thus, by the definition of $S_1$, $v_j$ is not adjacent to any vertex that corresponds to a literal that is set to true, and hence the clause $c_j$ is not satisfied by the assignment. Thus, for every truth assignment there is an unsatisfied clause, which means that $\Phi$ is unsatisfiable, as needed.

Conversely, suppose that $\Phi$ is unsatisfiable, and let $S_1$ and $S_2$ be disjoint independent sets of $G$. We show how to extend both of them to disjoint independent sets of size $n+1$.
The first step is to add a vertex from $V(K) \cup \{v\}$. For each $r \in \{1, 2\}$, if $S_r\cap V(K) = \emptyset$, there exists a vertex $v_{i_r}$ of $K$ such that $v_{i_r}$ is not adjacent to any vertex of $S_r$, since otherwise, by setting a literal to true whenever its corresponding vertex is a member of $S_r$, we can readily infer, by construction, that $\Phi$ is satisfied, which is a contradiction. Hence, if $S_1 \cap V(K) = \emptyset$, we can add $v_{i_1}$ to $S_1$ unless $v_{i_1} \in S_2$, in which case we add $v$ to $S_1$. The same argument holds for $S_2$.
So  we can assume that both $S_1$ and $S_2$ contain a vertex from $V(K) \cup \{v\}$. For each $i \in [n]$ and $r \in \{1, 2\}$ such that $S_r$ does not contain any vertex from $K^i$, observe that there are always at least two vertices from $K^i$ which are not adjacent to any vertex of $S_r$, hence we can add vertices so that both $S_1$ and $S_2$ contain a different vertex from each $K^i$.
Eventually, the resulting sets $S_1$ and $S_2$ are disjoint, independent and, by construction, of size $n+1$. This completes the hardness of deciding if $G$ is $\mathbf{W_2}$.

It remains to show that $\Phi$ is unsatisfiable if, and only if, $v$ is shedding. 

If $\Phi$ is satisfiable by some truth assignment, define a new variable $\ell_i = x_i$ if $x_i$ is set to true and $\ell_i = \overline{x_i}$ otherwise. Consider the independent sets $S = \{\ell_i^1\}_{1\leqslant i \leqslant n}$. Since the truth assignment satisfies $\Phi$, $S$ dominates $V(K) =N(v)$, and thus $v$ is not shedding.

Conversely, if $v$ is not shedding, there exists an independent set $S$ of $V(G) \setminus N[v]$ that dominates $N(v)=V(K)$. Up to the addition of some vertices, we assume that $S$ is maximal in $V(G) \setminus N[v]$, and thus contains exactly one vertex per $K^i$ for any $1\leqslant i \leqslant n$. Consider the truth assignment $\mu$ on $X$ defined by $\mu(\ell)=1$ if, and only if, $\ell^1\in S$ or $\ell^2 \in S$ for some literal $\ell$. Since any clause vertex is dominated by $S$, $\mu$ satisfies $\Phi$. This completes the proof.  
\end{proof}

\begin{proof}[Proof of Theorem \ref{thm:wk}]
From \cite{feghali2023three}, recognizing $\mathbf{W_k}$ graphs is in coNP.

The proof of hardness is essentially identical to that of Theorem \ref{thm:wcpartial}, but we include the details for completeness. We reduce from the complement of \textsc{3-Sat}.  Let $\Phi$ be a 3-CNF formula with a set $C$ of $m$ clauses $c_1,...,c_m$ and a set $X$ of $n$ variables $x_1,...,x_n$. Consider the graph $G$  that is obtained from $\Phi$ as follows:
\begin{itemize}
\item create a clique $K$ of size $m$ with vertex labels $v_1, \dots, v_m$;
\item create a clique $U$ of size $k-1$ with vertex labels $u_1, \dots, u_{k-1}$ and add an edge between each vertex of $U$ and each vertex of $K$;
\item for each variable $x_i$, add the vertices $x_i^1,\cdots x_i^k$ and $\overline{x_i}^1,\cdots \overline{x_i}^k$, and all possible edges between them;
\item if $x_i$ (resp. $\overline{x_i}$) appears in clause $c_j$, add the edges $(x_i^b ,v_j)$ (resp. $(\overline{x_i}^b ,v_j)$) for each $b \in [k]$.
\end{itemize}
For $1\leq i \leq n$, write $K^i = \{\ell^b\}_{\ell \in \{x_i, \overline{x_i}\}, b \in [k] }$ and note that $K^i$ is, by construction, isomorphic to the complete graph $K_{2k}$. 

\begin{claim}
$G$ is $\mathbf{W_{k-1}}$.
\end{claim}

\begin{proof}
As before, $\alpha(G) = n+1$. 
Now, let $A_1, \dots, A_{k-1}$ be disjoint independents sets of $G$. We must show that there exists pairwise vertex disjoint independents sets $S_1, \dots, S_{k-1}$ each of size $n+1$ such that  $A_i \subseteq S_i$ for $i \in [k-1]$. 

Since $U$ is of size $k-1$ and is not adjacent to any vertex of any $K^i$, we can extend each $A_r$ to make sure that $|A_r \cap (U \cup K)| = 1$ for every $r=1, \dots, k-1$. Now, for every $r \in [k-1]$, if $A_r \cap K = \emptyset$, then since each $K^i$ consists of $2k$ vertices, one can extend each $A_r$ so that $|A_r \cap K^i| = 1$. If $A_r \cap K = \{v_j\}$ for some $j \in [m]$, then since we can assume that each clause does not contain a literal and its negation, there must be $k$ vertices of each $K^i$ with no neighbor in $A_r$, hence we can extend $A_i$ to an independent set of size $n+1$.
\end{proof}

We now show that $\Phi$ is unsatisfiable if, and only if, $G$ is $\mathbf{W_k}$. 
First, suppose that $G$ is $\mathbf{W_k}$. For any truth assignment to $X$, for $1\leqslant i \leqslant n$, define a new variable $\ell_i$ by $\ell_i = x_i$ if $x_i$ is set to true, and by $\ell_i = \overline{x_i}$ otherwise. Let $S_r = \{u_r\}$ for $r \in \{1, \dots, k-1\}$ and $S_k = \{\ell_i^1\}_{1\leqslant i \leqslant n}$. Clearly, $S_1, \dots, S_{k-1}$ are pairwise disjoint independent sets in $G$. Since, by assumption, $G\in \mathbf{W_k}$, each $S_k$ is contained in an MIS $I_k$ of $G$. By our choice of the $S_j$'s, $I_k= S_k \cup \{v_j\}$ for some $j \in [m]$. Thus, by the definition of $S_k$, $v_j$ is not adjacent to any vertex that corresponds to a literal that is set to true, and hence the clause $c_j$ is not satisfied by the assignment. Thus, for every truth assignment there is an unsatisfied clause, which means that $\Phi$ is unsatisfiable, as needed.

Conversely, suppose that $\Phi$ is unsatisfiable, and let $S_1,\dots, S_k$ be disjoint independent sets of $G$. We show that there exists pairwise disjoint independents sets $T_1, \dots, T_k$ each of size $n+1$ such that  $S_i \subseteq T_i$ for $i \in [k]$. 

The first step is to add a vertex from $V(K) \cup V(U)$. Specifically, for $r \in \{1, \cdots, k\}$, if $S_r\cap V(K) = \emptyset$, then there exists a vertex $v_{i_r}$ of $K$ such that $v_{i_r}$ is not adjacent to any vertex of $S_r$, since otherwise, by setting a literal to true whenever its corresponding vertex is a member of $S_r$, we can readily infer, by construction, that $\Phi$ is satisfied, which is a contradiction. Hence, if $S_1 \cap V(K) = \emptyset$, we can add $v_{i_1}$ to $S_1$ unless $v_{i_1} \in \bigcup_{r>1} S_r$, in which case there exists $u_{j_1}\in V(U)$ such that $u_{j_1} \notin \bigcup_{r>1} S_r$. The same argument holds for $S_r$ for $r\in \{2,\cdots, k\}$.
In the following, we may assume that all $S_r$ contain a vertex from $V(K) \cup V(U)$. For each $i \in [n]$ and $r \in \{1, \cdots,k\}$ such that $S_r$ does not contain any vertex from $K^i$, observe that there are always at least two vertices from $K^i$ which are not adjacent to any vertex of $S_r$, hence we can add vertices so that all $S_r$ contain a different vertex from each $K^i$, as desired.
\end{proof}

Since $\mathbf{W_k} \subseteq \mathbf{E_s}$ for any $s, k \geqslant 1$, we have the following:

\begin{corollary}\label{cor:ex}
    Let $k\geqslant 2$ and $s\geqslant 1$ and $G$ be $\mathbf{E_s}$. Then deciding if $G$ is $\mathbf{W_k}$ is coNP-complete. 
\end{corollary}

We can also address the missing case in the statement of Corollary \ref{cor:ex} (with $k$ replaced by $1$), which requires a different proof.

\begin{theorem}\label{thm:eswc}
    Let $s \geqslant 1$ and $G$ be $\mathbf{E_s}$. Then deciding if $G$ is well-covered is coNP-complete. 
\end{theorem}

\begin{proof}
From \cite{CHVATAL1993179}, recognizing well-covered graphs is in coNP.
We reduce from the complement of \textsc{3-Sat}. Let $\Phi$ be a 3-CNF formula with a set $C$ of clauses $c_1,...,c_m$ and a set $X$ of variables $x_1,...,x_n$.  

We first test all $O(n^s \cdot 2^s)$ partial assignments of $s$ variables. If, for one of them, all clauses are satisfied, then it means that $\Phi$ is a satisfiable formula, and we output in this case a dummy negative instance of our problem.
In the following, we can thus assume that for any partial assignment of $s$ variables, there is at least one clause which is not (yet) satisfied.

Let $G$ be the graph obtained from $\Phi$ as follows:
\begin{itemize}
\item create a clique $K$ of size $m$ with vertex labels $v_1, \dots, v_m$;
\item for each variable $x_i$, add three new vertices $u_i, \overline{u_i}, w_i$ that form a triangle;
\item if $x_i$ (resp. $\overline{x_i}$) appears in clause $c_j$, add the edge $u_iv_j$ (resp. $\overline{u_i}v_j$).
\end{itemize}

\begin{claim}
    $G$ is $\mathbf{E_s}$.
\end{claim}
\begin{proof}
    It can be observed that $\alpha(G) = n+1$. 
Let $S \subseteq V(G)$ be an independent set of size $s$, and let $I = \{ i \mid S\cap \{u_i, \overline{u_i}, w_i\} = \emptyset\}$. If $S$ contains a vertex from $K$, then $S\cup \{w_i\}_{i\in I}$ is an independent set of size $n+1$ that contains $S$. Assume now that $S \cap K = \emptyset$. If there is a vertex $c$ in $K$ that has no neighbor in $S$, we add $c$ to $S$ and proceed as previously. Otherwise, assigning $true$ (resp. $false$) to every variable $x_i$ such that $u_i \in S$ (resp. $\overline{u_i} \in S$) leads to a partial assignment of at most $s$ variables such that all clauses are satisfied, which is impossible.
\end{proof}

Now, suppose $G$ is well-covered. For any truth assignment to the variables in $X$, we form an independent set $S$ in $G$ by adding to $S$ vertex $u_i$ if variable $x_i$ is set to true, and vertex $\overline{u_i}$ if variable $x_i$ is set to false. Note that $|S| = n$; moreover, as $G$ is well-covered and $\alpha(G) = n+1$, there is a vertex not in $S$ that has no neighbor in $S$. By construction, this vertex is $v_i$ for some $i \in [m]$, and so clause $c_i$ does not contain a true literal. Therefore, $\Phi$ is unsatisfiable.  

Conversely, suppose that $\Phi$ is  unsatisfiable. Let $S$ be any maximal independent set in $G$. By maximality, $|S\cap \{u_i, \overline{u_i}, w_i\}| =1$ for each $i \in [n]$.  Moreover, observe that $|S \cap K| = 1$ since if $S \cap K = \emptyset$ then, by setting to true the variables in $S \setminus \{w_i\}_{1\leqslant i \leqslant n}$, one obtains a (possibly partial) truth assignment  that satisfies $\Phi$, which is a contradiction. Thus $|S| = n+1$ as needed. 
\end{proof}

\section{Recognizing $\mathbf{E_s}$ graphs}\label{section:3}

In this section, we prove Theorem \ref{thm:es}. First, let us make a few observations. 
Recall that recognizing $1$-extendable graphs is NP-hard \cite{berge20231}, though the existence of an NP (or coNP) certificate is not clear. 
Moreover, its NP-hardness proof is via a reduction from \textsc{Exact MIS}, where, given a graph $G$ and an integer $k$, the problem is to determine if $\alpha(G)=k$, which, in turn, is NP-hard via a reduction from the classical  \textsc{Maximum Independent Set} problem. Let us note that \textsc{Exact MIS} belongs to the DP complexity class, introduced by Papadimitriou \cite{papadimitriou1982complexity}, and consists of problems which are the ``intersection'' of an NP problem and a coNP problem. Notice that this is different from being in NP and in coNP, since being DP-complete actually implies being both NP-hard and coNP-hard, which in turn implies not being in NP nor in coNP, unless NP = coNP. However, it appears that DP does not contain the problem of recognizing $1$-extendable graphs. Indeed, we show rather surprisingly that recognizing $1$-extendable graphs, and more generally $\mathbf{E_s}$ graphs, is $\Theta_2^p$-complete. The complexity class $\Theta_2^p = \text{P}^{\text{NP}[O(\log(n)]}$ refers to problems solvable by a polynomial-time deterministic algorithm using a logarithmic number of calls to a \textsc{Sat} oracle. This class has received attention, mainly because it contains various optimization problems across fields such as AI and voting schemes \cite{riege2006completeness,spakowski2000theta,wagner1987more,wagner1990bounded}. 
Interestingly, it also contains an approximation version of well-coveredness~\cite{hemaspaandra1998recognizing}, where one is interested in the graph class where the greedy algorithm for \textsc{Maximum Independent Set} provides a constant approximation ratio. We first show  that recognizing $\mathbf{E_s}$ graphs is in $\Theta_2^p$.

\begin{lemma}\label{lem:ex1}
    For any $s\geqslant 1$, recognizing $\mathbf{E_s}$ graphs is in $\Theta_2^p$.
\end{lemma}

\begin{proof}
We must describe an algorithm that solves the problem with only a logarithmic number of calls to a \textsc{Sat} oracle. To do so, for a fixed $s\geqslant 1$, we introduce the problem \textsc{Partial $s$-Extendable}, where, given a graph $G$ and an integer $r$, the task is to determine if each independent set of size at most $s$ can be extended to an independent set of size at least $r$. This problem is in NP, a certificate being the set of independent sets of $G$ of size $r$ that extend all independent sets of size at most $s$, and this certificate has a size of at most $O(r\cdot n^s)$. Therefore, \textsc{Partial $s$-Extendable} can be solved with a \textsc{Sat} oracle.

The algorithm recognizing $\mathbf{E_s}$ graphs operates as follows on input $G$:
\begin{enumerate}
    \item Compute $r=\alpha(G)$ using binary search, requiring $O(\log(n))$ \textsc{Sat} oracle calls to solve \textsc{Maximum Independent Set}.
    \item Solve \textsc{Partial $s$-Extendable} on input $(G,r)$ using only one \textsc{Sat} oracle.
\end{enumerate}
This algorithm straightforwardly recognizes $\mathbf{E_s}$ graphs in polynomial time, utilizing $O(\log(n))$ \textsc{Sat} oracle calls, thus the problem is in $\Theta_2^p$ as required. 
\end{proof}

Thus, it remains to establish hardness. To do so, we shall proceed by induction on $s$. To begin with, we need an complete problem for the class $\Theta_2^p$. An example of such a problem is \textsc{MIS Equality} which asks, given two graphs $G$ and $H$, if $\alpha(G)=\alpha(H)$. Notice that the equivalent problems for graph coloring, vertex cover and clique are also $\Theta_2$-complete \cite{wagner1987more,hemaspaandra1998recognizing}. In addition, we shall need the following construction (extracted from \cite{berge20231}).

For two disjoint graphs $G$ and $H$, let $\pi(G,H)$ be the graph obtained from $G$ and $H$ by adding  a complete bipartite graph $(\Pi_G := \{\pi_u \mid u \in V(G)\}, \Pi_H := \{\pi_u \mid u \in V(H)\})$, and adding the edge $u\pi_u$  for each $u\in V(G)\cup V(H)$, 
Notice that $\pi(G,H)$ has $2(|V(G)|+|V(H)|)$ vertices and $\alpha(\pi(G,H)) = \max(|V(G)|+\alpha(H), |V(H)|+\alpha(G))$. Let us first address the base case $s=1$. 
\begin{lemma}\label{lem:ex2}
    Recognizing $1$-extendable graphs is $\Theta_2^p$-hard.
\end{lemma}

\begin{proof}

    To prove the lemma, we reduce from \textsc{MIS Equality}. 
    
    Let $(G, H)$ be an instance of \textsc{MIS Equality}, where up to the addition of some universal vertices, both $G$ and $H$ have the same number $n$ of vertices. Let $G':= \pi(G,H)$, so that $G'$ has $4n$ vertices and $\alpha(G') = n+\max(\alpha(G), \alpha(H))$. We show that $G'$ is $1$-extendable if and only if $\alpha(G) = \alpha(H)$.

    Suppose that $\alpha(G)=\alpha(H) =\alpha$. For each vertex $v\in V(G)$, set $S_v = \{v\} \cup (\Pi_G \setminus \{\pi_v\}) \cup S_H$, where $S_H$ is any MIS of $H$. Then $S_v$ has cardinality $n+\alpha$ and thus is an MIS of $G'$. Notice that $\Pi_v \cup S_H$ is also an MIS of $G'$, and thus $\pi_v$ belongs to an MIS too. By symmetry, all vertices from $V(H) \cup \Pi_H$ belong to an MIS of $G'$, thus $G'$ is $1$-extendable.

    Now, suppose that $G'$ is $1$-extendable. Then either $\alpha(G') = \alpha(G) + n$ or $\alpha(G') = \alpha(H) + n$.  Assume  that $\alpha(G') = n+\alpha(G)$. Let $v\in V(G)$.  By $1$-extendability of $G'$, $\pi_v$ is in an MIS $S$ of $G'$. Let $S_G = (V(G) \cup \Pi_G)\cap S$ and $S_H = (V(H) \cup \Pi_H)\cap S$. Since $\pi_v\in S$ and $S$ is independent, $S$ contains no vertex from $\Pi_H$. It follows that $S_H$ is an MIS of $H$ of cardinality $\alpha(H)$. In addition, $S_G$ is an MIS of $G'[V(G)\cup \Pi_v]$ and thus has cardinality $n$. Altogether, $|S|=n+\alpha(H) = \alpha(G')$, which gives $\alpha(G) = \alpha(H)$. This completes the proof. 
\end{proof}

\begin{corollary}[\cite{berge20231}]
    Recognizing $1$-extendable graphs is both NP-hard and coNP-hard, and thus  neither in NP nor in coNP unless NP=coNP.
\end{corollary}

To address the general case $s\geq 1$, we require one more lemma. Given a graph $G$, let $G^+=\pi(G_1,G_2)$ where $G_1$ and $G_2$ are two copies of $G$. Notice that $\alpha(G^+)=|V(G)|+\alpha(G)$. For brevity, we write $\Pi_1=\Pi_{G_1}$ and $\Pi_2=\Pi_{G_2}$ in $G^+$.

\begin{lemma}\label{lemma:G+}
    For any graph $G$ and any $s\geqslant 1$, $G^+\in \mathbf{E_s}$ if and only if $G\in \mathbf{E_{s-1}}$. 
\end{lemma}

\begin{proof}
    Suppose first that $G\in \mathbf{E_{s-1}}$. Let $A\subseteq V(G^+)$ be an independent set of size at most $s$. We must show that $A$ is contained in an MIS of $G^+$.  
    
    Suppose $A\cap (\Pi_1\cup \Pi_2) = \emptyset$. If $A\subseteq V(G_1)$, then $A$ can be extended into an independent set of size $|V(G)|$ by adding only vertices of $\Pi_1$, and then, in turn, into an independent set of size $|V(G)|+\alpha(G)$ by adding any MIS of $G_2$. Similarly, if $A\subseteq V(G_2)$ then we are done. So we can assume that  $A$ intersects both $V(G_1)$ and $V(G_2)$ (notice that this case does not happen if $s=1$). As $|A| \leq s$, we have that $|A\cap V(G_1)|\leqslant s-1$. Since $G_1\in \mathbf{E_{s-1}}$, $A\cap V(G_1)$ can be extended into an MIS $\tilde{A_1}$ of $G_1$. At the same time, $A\cap V(G_2)$ can be extended to an independent set $\tilde{A_2}$ of size $|V(G)|$ using only new vertices from $\Pi_2$. Notice that $\tilde{A_1} \cup \tilde{A_2}$ is an independent set of size $|V(G)|+\alpha(G)$ that contains $A$.

    So we can assume that $A\cap (\Pi_1\cup \Pi_2) \neq \emptyset$. Since there are all possible edges between $\Pi_1$ and $\Pi_2$ in $G^+$, trivially $A$ intersects at most one of $\Pi_1$ and $\Pi_2$, say $\Pi_1$. Let $A_1= A\cap (V(G_1) \cup \Pi_1)$, and $A_2 = A\cap V(G_2)$. As in the preceding paragraph, $A_1$ can be extended to an independent set $\tilde{A_1}$ of size $|V(G)|$ using only new vertices from $\Pi_1$, and since $|A_2| \leqslant s-1$ and $G_2\in \mathbf{E_{s-1}}$, $A_2$ can be extended to an independent set $\tilde{A_2}$ of size $\alpha(G)$ using only vertices from $V(G_2)$. Then $\tilde{A_1} \cup \tilde{A_2}$ is an MIS of $G^+$ that contains $A$. Therefore, $G^+$ is $\mathbf{E_s}$.

    Conversely, suppose $G^+\in \mathbf{E_s}$. Let $A\subseteq V(G)$ be an independent set of $G$ of size at most $s-1$. We must show that $A$ extends into an MIS in $G$.  Let $v\in \Pi_2$, and consider the independent set $A_1\cup \{v\}$ where $A_1$ is exactly $A$ in the copy $G_1$ of $G$ in $G^+$. Since $A_1\cup\{v\}$ is an independent set of $G^+$ of size at most $s$, it can be extended to an MIS $S$ of $G^+$, which has size $|V(G)|+\alpha(G)$. Since $v\in S$, $S\cap \Pi_1 = \emptyset$. Thus, $S$ can be partitioned into two sets $S_1$ and $S_2$, where $S_1$ induce a maximum independent set of $G_1$ and $S_2$ a maximum independent set of $G^+[V(G_2) \cup \Pi_2]$. We notice that $|S_2|=|V(G)|$, and thus $|S_1|=\alpha(G)$. In addition, $A_1\subseteq S_1$ and so $A$ can be extended to an MIS of $G$.
\end{proof}

We are now ready to prove Theorem \ref{thm:es}; in fact, we prove the following more general result. 

\begin{theorem}\label{thm:EsE(s-1)}
    Let $s\geqslant 2$ and $G$ be $\mathbf{E_{s-1}}$. Then deciding whether $G$ is $\mathbf{E_s}$  is $\Theta_2^p$-complete.
\end{theorem}

\begin{proof}
    By Lemma \ref{lem:ex1}, the problem is in $\Theta_2^p$. To show hardness, we proceed by induction on $s$. For $s = 1$,  by Lemma \ref{lem:ex2}, $\mathbf{E_1}$ graphs is $\Theta_2^p$-complete on arbitrary (that is, $\mathbf{E_0}$) graphs. So we can assume that $s \geq 2$ and that, given a graph $G$ in  $\mathbf{E_{s-2}}$, deciding whether $G$ is $\mathbf{E_{s-1}}$ is $\Theta_2^p$-hard. 

    Now, since $G\in \mathbf{E_{s-2}}$, it follows from Lemma \ref{lemma:G+} that $G^+\in \mathbf{E_{s-1}}$ and that $G\in \mathbf{E_{s-1}}$ if and only if $G^+\in \mathbf{E_s}$. This implies the result. 
\end{proof}

As said in the introduction, $s$-extendable graphs and $\mathbf{E_s}$ are slightly different. Nevertheless, the same proof, word by word, implies

\begin{corollary}
    Let $s \geq 1$. Recognizing $s$-extendable graphs is $\Theta_2^p$-complete.
\end{corollary}

\section{Chordal graphs}
Let us recall some usual definitions and facts.

A vertex $v$ is \textit{simplicial} in a graph $G$ if $G[N(v)]$ is a complete graph. If $v$ is simplicial, we say that $N[v]$ is a simplex. A \textit{chordal graph} is a graph with no induced cycle of size at least four. By a celebrated and folklore theorem of Dirac, a graph $G$ is chordal if and only if (iff) every induced subgraph of $G$ has a simplicial vertex. Phrased differently, $G$ is chordal iff there exists an ordering $v_1, \dots, v_n$ of $V(G)$, called a \emph{perfect elimination ordering}, such that for $i \in [n]$, the graph induced by $\{v_j: v_iv_j \in E(G), j < i\}$ is complete. 
Another classical and convenient object related to a chordal graph is its \textit{clique tree decomposition}. For any chordal graph, its clique tree decomposition is a tree whose vertex set is in bijection with set of maximal cliques of $G$, and such that for every vertex of $G$, the vertices of the tree corresponding to cliques containing this vertex induce a connected subtree. It is a folklore result that such object can be found in linear time.

\subsection{Recognizing $\mathbf{W_k}$ chordal graphs}
In this subsection, we show that deciding whether a chordal graph is $\mathbf{W_k}$ can be done in linear time for each $k \geq 1$. Recall that the case $k = 1$, that is, the well-coveredness of chordal graphs, has already been studied before, showing that a chordal graph is well-covered if, and only if, each vertex belongs to exactly one simplex. Equivalently, the following result was established.

\begin{theorem}[\cite{Prisner}]
    A chordal graph $G$ is well-covered if, and only if there exists a partition of $V(G)$ into simplices.
\end{theorem}

The previous theorem directly leads to an algorithm that decides if a chordal graph is well-covered. The naive algorithm that decides if a given chordal graph $G$ is $\mathbf{W_k}$ consists in checking, for any subset of $k-1$ vertices $V' \subseteq V(G)$, whether $G-V'$ is well-covered. The complexity of this naive algorithm is $O(n^k(n+m))$. This running time can be improved by generalizing the previous characterization of well-covered chordal graphs to $\mathbf{W_k}$ chordal graphs.

\begin{proposition}
    Let $k \geqslant 1$. A chordal graph $G$ is $\mathbf{W_k}$, if and only if, if there exists a partition of $V(G)$ into simplices, each of them having at least $k$ simplicial vertices.
\end{proposition}

\begin{proof}

    Recall that from Staples, a graph $G$ is $\mathbf{W_k}$ if, and only if, for any set $S\subseteq V(G)$ of $k-1$ vertices, $G-S$ is well-covered and $\alpha(G)=\alpha(G-S)$.
    
    If a graph $G$ is $\mathbf{W_k}$, it is indeed well-covered, and thus there exists a partition $V_1,...,V_q$ of $V(G)$ such that each $V_i$ induces a simplex. By contradiction, assume that $V_1$ has less than $k$ simplicial vertices. Let $S$ be the set of simplicial vertices of $V_1$, we have $|S| \leqslant k-1$. Then, notice that $G-S$ is not well-covered using the characterization of chordal well-covered graphs. Indeed, the vertices of $V_1\setminus S$ are not in any simplex.

    Reciprocally, if there exists a partition of $V(G)$ into $V_1,...,V_q$ such that each $V_i$ is a simplex with at least $k$ simplicial vertices, then after removing any $k-1$ vertices of $G$, there still exists a partition of $V(G)$ into simplices, and there is still an MIS of size $q$.
\end{proof}

The previous theorem leads to a linear algorithm to check if a chordal graph is $\mathbf{W_k}$. Indeed, one can find the partition into simplices in time $O(n+m)$, and then check if each simplex has at least $k$ simplicial vertices in time $O(n+m)$.

\subsection{Recognizing $\mathbf{E_s}$ chordal graphs}\label{sec:chordal}

In this subsection, we prove Theorem \ref{thm:chordalmain}.
Let us first establish the first part, which consists in the characterization.

\begin{theorem}
Let $G$ be a chordal graph. $G$ is $1$-extendable iff there is a partition of $V(G)$ into $\alpha(G)$ parts such that each of them is a maximal clique in $G$.
\end{theorem}
\begin{proof}
Let $G$ be a chordal graph.
Assume first that $G$ is $1$-extendable.
We prove by induction on $\alpha(G)$ that $V(G)$ can be partitioned into $\alpha(G)$ maximal cliques, the statement being obvious if $\alpha(G) = 1$.
In the following, we assume $\alpha(G) > 1$. 
Let $v$ be a simplicial vertex of $G$ which belongs to a leaf $C_v$ of the tree decomposition of $G$ (it is thus a maximal clique). Let $G' = G[V(G) \setminus C_v]$. Observe that $\alpha(G') = \alpha(G)-1$, hence we can apply the induction hypothesis on $G'$ and obtain a partition of $V(G')$ into $C_1$, $\dots$, $C_q$ with $q = \alpha(G)-1$. We claim that for every $i \in [q]$, $C_i$ is also a maximal clique in $G$. Indeed, if there is $u \in C_v$ such that $C_i \cup \{u\}$ is also a clique, then $u$ would not belong to any maximum independent set in $G$ which violates its $1$-extendability, since any maximum independent set of $G$ must intersect $C_i$.

We also prove the other direction by induction on $\alpha(G)$. Assume that $V(G)$ can be partitioned into $C_1$, $\dots$, $C_{\alpha(G)}$, where each $C_i$ is a maximal clique. In the tree decomposition of $G$, we choose an arbitrary node of degree at least $2$ to be the root, and w.l.o.g. we may assume that $C_1$ is a leaf of the tree decomposition. Let $P$ be the maximal clique of $G$ which is the parent of $C_1$ in the tree decomposition. We denote by $v$ any simplicial vertex of $C_1$. 
Let $G' = G[V(G) \setminus C_1]$. We have $\alpha(G') = \alpha(G)-1$, and by induction $G'$ is $1$-extendable. It remains to show that every vertex of $C_1$ belongs to a maximum independent set of $G$. This is clearly the case for $v$ since it is only adjacent to vertices in $C_1$. This is also the case for every $u \in C_1 \setminus \{v\}$ since otherwise it would mean that $u$ is adjacent to every vertex of $P$, but then $P$ would not be a maximal clique in $G$.
\end{proof}

 Deciding if a given chordal graph is $k$-extendable can naively be done in time $O(n^{k}(n+m))$ by enumerating all independent sets of $G$ of size $k$ in time $O(n^k)$ and by checking the independence number of the graph when removing such an independent set. 
In particular, with this method the complexity for testing $1$-extendability of a chordal graph can be done in $O(n(n+m))$ time. The previous characterization allows to do it in linear time. 
 We first need the following observation.

\begin{lemma}\label{lemma : 1extSimplicial}
    If a chordal graph $G$ is $1$-extendable and $v$ is a simplicial vertex, then $G-N[v]$ is $1$-extendable.
\end{lemma}

\begin{proof}
    First, notice that $\alpha(G-N[v]) = \alpha(G)-1$. Let $u\in V(G) \setminus N[v]$, and since $G$ is $1$-extendable, there exists an MIS $S$ of $G$ such that $u\in S$. If $S$ does not contain any vertex from $N[v]$, then $S\cup \{v\}$ is an independent set of size $\alpha(G)+1$, which is not possible. In addition, since $N[v]$ induces a clique, $S$ contains exactly one vertex of $N[v]$. Thus, $S\setminus N[v]$ is a maximum independent set of $G-N[v]$ that contains $u$.
\end{proof}

We are now ready to complete the proof of Theorem \ref{thm:chordalmain}.

\begin{theorem}
    There exists an algorithm that decide if a given chordal graph is $1$-extendable running in time $O(n+m)$.
\end{theorem}

\begin{proof}
We describe an algorithm that, given a perfect elimination ordering of a chordal graph, decides if it is $1$-extendable and if it is the case, returns a partition of the vertex set into maximal cliques.

    Let $G$ be a chordal graph, and consider a perfect elimination ordering $v_1,...,v_n$ of $G$. First, we check in time $O(n+m)$ if $G$ is a clique, and thus if $\alpha(G)=1$. Otherwise, since $v_1$ is simplicial in $G$, $\alpha(G-N[v_1])=\alpha(G)-1$ and $G-N[v_1]$ has to be $1$-extendable by Lemma \ref{lemma : 1extSimplicial}. We run the algorithm on $\{v_1,...,v_n\} \setminus N[v_1]$, which checks if $G-N[v_1]$ is indeed $1$-extendable. If not, $G$ cannot be $1$-extendable, and otherwise, we obtain the clique partition $C_2,...,C_q$ in time $(n-d(v_1)) +(m-\sum_{u\in N[v_1]} d(u))$. 
    Now, observe that $\{N[v], C_2, \dots, C_q\}$ is a partition of $V(G)$ into maximal cliques unless some vertex $v \in N(v)$ is adjacent to all vertices of some $C_i$, but in that case $v$ cannot belong to any independent set of $G$ (since $\alpha(G) = q+1$), and hence $G$ is not $1$-extendable. This test can be performed in $O(m)$ time, and thus the overall complexity is $O(n+m)$.
\end{proof}

However, contrary to $\mathbf{W_k}$, this algorithm cannot be extended to a linear-time algorithm for recognizing $\mathbf{E_s}$ graphs for any $s \geqslant 2$, as shows the following rather surprising result.

\begin{theorem}
Given a chordal graph $G$ and an integer $s \geqslant 1$, the problem of deciding whether $G$ is in $\mathbf{E_s}$ is coW[2]-hard when parameterized by $s$.
\end{theorem}

\begin{proof}
    We reduce from the complement of \textsc{Minimum Dominating Set}. Let $(G, s)$ be the input, with $|V(G)| = n$. We may assume that $G$ does not contain a universal vertex, otherwise the problem is trivial. Construct $G'$ as follows: $V(G')$ is composed of a set $C = \{c_v : v \in V(G)\}$ inducing a clique together with two sets $I = \{i_{v} : v \in V(G)\}$ and $I' = \{i'_{v} : v \in V(G) \}$ inducing two independent sets, respectively. For every $uv \in E(G)$, we connect $i_{v}$ to $c_v$, $c_u$ and $i'_{v}$. It can be observed that $G'$ is indeed a chordal graph ($G'[C \cup I]$ is a split graph, and $I'$ is a set of vertices of degree one).
    Moreover, $G'$ is $1$-extendable, as $\alpha(G') = n+1$, and each vertex belongs to an independent set of size $n+1$ since $G$ does not contain a universal vertex. In particular, observe that any maximum independent set of $G'$ contains a vertex of $C$.
    We prove that $G' \notin \mathbf{E_s} $ if, and only if $G$ contains a dominating set of size at most $s$.
      
    If $G'\notin \mathbf{E_s}$, then there exists a set $S \subseteq V(G')$ of size at most $s$ that is not contained in any independent set of size $n+1$. If $S \cap C \neq \emptyset$, then let $S' = S \cup \{i'_v \mid i_v \notin S\}$, and notice that $|S'|=n+1$ and thus $S'$ is an MIS that contains $S$, which is not possible. So we may assume that $S \cap C = \emptyset$. Let $S'=S \cup \{i'_v \mid i_v \notin S\}$ which an independent set that contains $S$ and for any $v\in V(G)$, $S'$ contains either $i_v$ or $i'_v$. Since $S'$ cannot be extended to an MIS, it implies that $S'$ dominates $C$, which implies that $S \cap I$ is in one-to-one correspondence with a dominating set in $G$ of at most $s$.

    Conversely, assume that $G'\in \mathbf{E_s}$, and let $D$ be any subset of $V(G)$ of size at most $s$. Then the set $S = \{i_v : v \in D\}$ is an independent set of size at most $s$ in $G'$ which is thus contained in an MIS $S'$ of $G'$, which necessarily intersects $C$ on, say, $c_{v^*}$. It implies that $i_{v^*}$ does not belong to $S$ and that $c_{v^*}$ is not adjacent to any vertex of $S$, which means that $v^*$ does not belong to $D$ nor is dominated by $D$ in $G$, and thus $D$ is not a dominating set in $G$.
\end{proof}

\section{Conclusion and open questions}\label{sec:conclusion}
We extended the $\mathbf{W_k}$-hierarchy and investigated relative computational complexities within this hierarchy. We addressed open questions by demonstrating that recognizing $\mathbf{W_k}$ graphs and shedding vertices are coNP-complete on well-covered graphs. We closed the complexity gap of testing $1$-extendability, proving that the problem is $\Theta_2^p$-complete. Additionally, we provided a simple characterization of $1$-extendable chordal graphs, leading to a linear-time algorithm to decide if a graph is $1$-extendable, addressing questions left open in \cite{berge20231}.
Up to our knowledge, several questions remain open, with two of particular interest.

\begin{problem}
What is the computational complexity of recognizing well-covered triangle-free graphs?
\end{problem}

This problem was indirectly raised by Plummer in his survey \cite{Plummer1993}, where he asked if there exists a characterization of well-covered or $\mathbf{W_2}$ graphs of girth at least $4$. In fact, the complexity of recognizing well-covered $K_t$-free graphs is still open for any constant $t\geqslant 3$, and even the same problem on graphs such that $\chi(G)=t$ for any fixed $t \geqslant 3$ \cite{ALVES201836}. Notice that recognizing $\mathbf{E_s}$ triangle-free graphs is also $\Theta_2^p$-complete. In the instance $(G,H)$ of \textsc{MIS Equality}, we can assume that both graphs are triangle-free, up to an even subdivision of the edges. Then, one can check that every reduction does not create any triangle.
A similar problem, raised by Hujdurovi{\'c}, Milani{\v{c}}, Martin and Ries, concerns co-well-covered graphs, \textit{i.e.}, graphs for which the complement is well-covered.

\begin{problem}[\cite{hujdurovic2018graphs}]
What is the computational complexity of recognizing graphs which are both well-covered and co-well-covered?
\end{problem}

Both problems are linked, as any triangle-free graph with no isolated vertex is co-well-covered. Thus, if recognizing triangle-free well-covered graphs is coNP-complete, the same holds for the second problem. Conversely, if there exists a polynomial-time algorithm to check if a graph is well-covered and co-well-covered, one can recognize triangle-free well-covered graphs in polynomial time.

We conclude with two problems related to extendability. The first concerns $1$-extendable graphs with high girth. While there is a characterization of well-covered graphs of girth $5$, recognizing $1$-extendable graphs is NP-hard even for arbitrarily large girth \cite{berge20231}. Can we find some structural property for $1$-extendable graphs of large girth, such as the existence of a linear-size independent set (which is the case for well-covered graphs of girth $5$)? The second concerns the computational complexity of recognizing $\mathbf{E_s}$ graphs when $s$ is part of the input. The problem is $\Theta_2^p$-hard but not complete. One can show that it belongs to $\Pi_2^p$ in the polynomial hierarchy; it would be interesting to close this complexity gap.

\bibliographystyle{plain}
\bibliography{biblio.bib}

\end{document}